\documentclass[11pt]{article}

\usepackage[T1]{fontenc}
\usepackage{lmodern}

\usepackage{amsmath, amsfonts, amssymb, amsthm} 
\usepackage{mathtools}

\usepackage[margin=2.5cm]{geometry}
\usepackage{color, graphicx,cite}
\usepackage{booktabs,diagbox}
\usepackage{overpic}
\usepackage[shortlabels]{enumitem}

\usepackage{makecell}
\usepackage{tikz}

\newcommand\diag[4]{%
  \multicolumn{1}{p{#2}|}{\hskip-\tabcolsep
  $\vcenter{\begin{tikzpicture}[baseline=0,anchor=south west,inner sep=#1]
  \path[use as bounding box] (0,0) rectangle (#2+2\tabcolsep,\baselineskip);
  \node[minimum width={#2+2\tabcolsep-\pgflinewidth},
        minimum  height=\baselineskip+\extrarowheight-\pgflinewidth] (box) {};
  \draw[line cap=round] (box.north west) -- (box.south east);
  \node[anchor=south west] at (box.south west) {#3};
  \node[anchor=north east] at (box.north east) {#4};
 \end{tikzpicture}}$\hskip-\tabcolsep}%
}

\usepackage{etoolbox} 
\makeatletter
\patchcmd{\@maketitle}{\LARGE \@title}{\LARGE\bfseries\@title}{}{}

\renewcommand{\@seccntformat}[1]{\csname the#1\endcsname.\quad}
\makeatother

\definecolor{NiceBlue}{rgb}{0.2,0.2,0.75}
\newcommand{\struc}[1]{{\color{NiceBlue} #1}}

\definecolor{darkblue}{rgb}{0,0,.5}
\usepackage{hyperref}
\hypersetup{
	colorlinks=true,        % false: boxed links; true: colored links
	linkcolor=darkblue,     % color of internal links
	citecolor=darkblue,     % color of links to bibliography
	urlcolor=darkblue       % color of external links
}

\makeatletter
\def\th@plain{%
	\thm@notefont{}% same as heading font
	\itshape % body font
}
\def\th@definition{%
	\thm@notefont{}% same as heading font
	\normalfont % body font
}

\renewenvironment{proof}[1][\proofname]{\par
	\normalfont
	\topsep0\p@\@plus3\p@ \trivlist
	\item[\hskip\labelsep\itshape
	#1\@addpunct{.}]\ignorespaces
}{%
	\qed\endtrivlist
}
\makeatother

\newtheorem{theorem}{Theorem}[section]
\newtheorem{lemma}[theorem]{Lemma}
\newtheorem{corollary}[theorem]{Corollary}
\newtheorem{proposition}[theorem]{Proposition}
\theoremstyle{definition}
\newtheorem{definition}[theorem]{Definition}
\theoremstyle{definition}
\newtheorem{example}[theorem]{Example}
\theoremstyle{definition}
\newtheorem{remark}[theorem]{Remark}

\allowdisplaybreaks

\newcommand{\Id}{\ensuremath{\operatorname{Id}}}
\newcommand{\Fix}{\ensuremath{\operatorname{Fix}}}
\newcommand{\Ker}{\ensuremath{\operatorname{Ker}}}
\newcommand{\Int}{\ensuremath{\operatorname{int}}}
\newcommand{\bdry}{\ensuremath{\operatorname{bdry}}}
\newcommand{\cone}{\ensuremath{\operatorname{cone}}}
\newcommand{\dist}{\ensuremath{\operatorname{dist}}}

\newcommand{\N}{\mathbb{N}}
\newcommand{\R}{\mathbb{R}}
\newcommand{\with}{\ : \ }

\begin{document}

\title{Douglas--Rachford is the best projection method}

\author{Minh N. Dao\thanks{School of Sciences, RMIT University, Melbourne, VIC 3000, Australia.}, Mareike Dressler\thanks{School of Mathematics and Statistics, University of New South Wales, Sydney, NSW 2052, Australia.}, Hongzhi Liao$^\dagger$, and Vera Roshchina$^\dagger$}

\maketitle

\begin{abstract}
We prove that the Douglas--Rachford method applied to two closed convex cones in the Euclidean plane converges in finitely many steps if and only if the set of fixed points of the Douglas--Rachford operator is nontrivial. We analyze this special case using circle dynamics. 

We also construct explicit examples for a broad family of projection methods for which the set of fixed points of the relevant projection method operator is nontrivial, but the convergence is not finite. This three-parametric family is well known in the projection method literature and includes both the Douglas--Rachford method and the classic method of alternating projections. 

Even though our setting is fairly elementary, this work contributes in a new way to the body of theoretical research justifying the superior performance of the Douglas--Rachford method compared to other techniques. Moreover, our result leads to a neat sufficient condition for finite convergence of the Douglas--Rachford method in the locally polyhedral case on the plane, unifying and expanding several special cases available in the literature.  
\end{abstract}

{\small

}

\section{Introduction}

Projection methods, such as the method of alternating projections and the Douglas--Rachford method, primarily serve the purpose of finding a point in the intersection of two (closed convex) sets using an interative process built by composing the identity mapping, projection and reflection operators and their convex combinations. For instance, the method of alternating projections \cite{Neu50} is an iterative application of the composition of projections onto two (different) sets, and each step of the Douglas--Rachford method \cite{DR56,LM79} is the composition of two reflections averaged with the identity, applied to the preceding iterate. The aforementioned methods belong to a broader family of projection methods defined in \cite{DP18}. Even though there exist other popular projection methods that do not belong to this family, such as Dykstra's method \cite[Section~30.2]{BC17} and the method of circumcentred reflections \cite{BCS18}, for our purposes it is reasonable to limit the scope to a specific family of methods.

Our goal is to show that the Douglas--Rachford method is ``the best'' method in this family since it converges finitely under very mild assumptions, moreover, the number of iterations needed to reach the set of fixed points has a uniform bound that does not depend on the starting point (see Theorem~\ref{thm:main-result}). This finite convergence property is not true for all other members of this family: we demonstrate this by providing examples for every other case (see Theorem~\ref{thm:badmethods}). We restrict our setting to the case of two closed convex cones in the two-dimensional space. This setting is sufficient to showcase the differences in the behavior of various projection methods, and allows for a neat proof of the finite convergence of the Douglas--Rachford method using circle dynamics.

The finite convergence of the Douglas--Rachford method has been investigated in \cite{BDNP16} for the case involving an affine subspace and a locally polyhedral set, as well as for a hyperplane and an epigraph. It is then proved in \cite{ABT16} for the scenario of a finite set and a halfspace. Subsequently, some more finite-convergence results for the Douglas--Rachford method are provided in \cite{BD17}. However, most of the existing sufficient conditions for finite convergence require the presence of Slater's condition. 
Only one noteworthy exception is when one set is a line or a halfplane, and the other is a locally polyhedral set in the Euclidean plane, where finite convergence can still be established without Slater's condition. The key convergence result of this paper allows us to obtain a general condition for the finite convergence of the Douglas--Rachford method for locally polyhedral sets in the plane subsumes the known results; note that in this case we may not be able to obtain a uniform bound on the number of iterations necessary to reach the set of fixed points.
It is also worth mentioning that finite convergence of the method of alternating projections has recently been studied in \cite{BBS21,BLM21}, but only for two non-intersecting sets.

\medskip
The paper is organized as follows. In Section~\ref{sec:mainresults} we explain our setting, state and briefly discuss the main results (Theorems~\ref{thm:main-result} and~\ref{thm:badmethods}). While Theorem~\ref{thm:main-result} characterizes the finite convergence of the Douglas--Rachford method for the case of two cones in the plane, Corollary~\ref{cor:cones} provides a more convenient condition for the finite convergence, requiring the intersection of two sets to have at least two points in their intersection. We state and prove a general finite convergence result for locally polyhedral sets in the plane in Corollary~\ref{cor:polys}. In Section~\ref{sec:finite} we build the auxiliary theory needed for the proof of Theorem~\ref{thm:main-result}, demonstrating that Douglas--Rachford converges finitely for the case of two convex cones under natural conditions. Section~\ref{sec:counterexamples} is dedicated to the detailed list of examples needed to prove Theorem~\ref{thm:badmethods}, that shows that Douglas--Rachford is the only method in the broad family of projection methods that enjoys the finite-convergence properties given by Theorem~\ref{thm:main-result}.

\medskip
We work in the real finite-dimensional space, predominantly limiting ourselves to $\R^2$. Whenever our proofs work in a more general setting than $\R^2$, we state the results in suitable generality. In consistency with the literature on projection methods that is often set in Hilbert spaces, we use \struc{$\langle x,y\rangle $} to denote the standard inner product on $\R^d$ and $\struc{\|x\|} = \sqrt{\langle x,x\rangle}$ for the norm induced by the inner product. For any \(x,y \in \R^d\), the line segment connecting $x$ and $y$ is given by
\[
\struc{[x,y]}\coloneqq \{\lambda x+(1-\lambda)y \with \lambda \in [0,1]\}.
\]
We say that $r\subseteq \R^d$ is a \struc{\emph{ray}} if there exists some $x\in \R^d$ such that 
\[
\struc{r} = \{\lambda x \with \lambda\in (0, +\infty)\}.
\]
A set $S\subseteq \R^d$ is called a \struc{\emph{cone}} if for all $x\in S$ and all $\lambda\in (0, +\infty)$, we have $\lambda x \in S$. Any cone is a union of rays. 
For $S\subseteq \R^d$, its \struc{\emph{polar cone}} $S^{\circleddash}$ is defined by  
\[
\struc{S^\circleddash} \coloneqq \{y\in \R^d \with \text{for all } x\in S,\quad \langle x,y\rangle \leq 0\}
\]
and its \struc{\emph{orthogonal complement}} $S^{\perp}$ is given by
\[
\struc{S^\perp} \coloneqq \{y\in \R^d \with \text{for all } x\in S,\quad \langle x,y\rangle = 0\}.
\]
We denote the \struc{\emph{interior}} of $S$ by \struc{$\Int S$} and its \struc{\emph{boundary}} by \struc{$\bdry S$}.

\section{Generalized Projection Methods and the  Main Results}\label{sec:mainresults}

Recall that for a closed convex set $C\subseteq \R^d$, by \cite[Theorem 3.16]{BC17}, the projection $\struc{P_C(x)}$ of $x\in \R^d$ onto $C$ is the unique point $p = P_C(x)\in C$ such that, for all $u\in C$, 
\[
\|x-p\| \leq \|x-u\|.
\]

Given two closed convex sets $A,B\subseteq \R^d$ and an initial point $x_0\in \R^d$, the intersection of $A$ and $B$ can be approached by alternatively projecting the current iterate onto each one of these sets. We may consider the sequence of two projections as one iteration and treat this method as fixed point iterations of the composition $\struc{T_{AP}(x)} \coloneqq P_B \circ P_A (x)$, called \struc{\emph{the method of alternating projections}}. This method, also called \emph{von Neumann method} after its inventor, can have very slow convergence with rate depending on the relative geometry of the two sets (exemplified by the case of two linear subspaces where the rate of convergence is linear and depends on the Friedrichs angle between the subspaces; see \cite{KW88,BGM10,BGM11}).

Amongst many ways to improve the convergence rate of the method of alternating projections, one can `overproject' by pushing the projection beyond the boundary of the closed convex set. An extreme version of such an overprojection is a reflection in the hyperplane perpendicular to the direction of the projection. 

For a closed convex set $C\subseteq \R^d$, we define the \struc{\emph{reflection operator}} $R_C(x)$ as 
\[
\struc{R_C(x)} \coloneqq 2 P_C(x)-x,
\]
that is, $R_C = 2 P_C-\mathrm{Id} $, with \struc{$\mathrm{Id}$} being the \struc{\emph{identity operator}}.

The \struc{\emph{Douglas--Rachford operator}} on an ordered pair of closed convex sets $A,B\subseteq \R^d$ is the composition of two reflections, averaged with the identity operator,  
\[
\struc{T_{DR}} \coloneqq \frac{1}{2}({\rm Id} +R_B\circ R_A).
\]

Analogous to overprojecting, it can be useful to underproject, and to compose the fixed point operator from various combinations of such mappings. It is reasonable to build such operators preserving the benign properties that help control and track the convergence. We therefore use the following definition of a generalized projection method from \cite{DP18,DP19} that allows us to work in a fairly general setting, while ensuring converegence of our method.

\begin{definition}[Generalized projection operator]\label{def:projop}
Let \(A, B \subseteq \R^d\) be closed convex sets, and $\mu,\lambda \in (0,2], \kappa \in (0,+\infty)$. Denote 
\[
\struc{P_A^{\lambda}} \coloneqq (1-\lambda)\Id + \lambda P_A \quad{\rm and}\quad \struc{P_B^{\mu}} \coloneqq (1-\mu)\Id + \mu P_B.
\]
We define the \struc{\emph{generalized projection operator}} as
\begin{equation*}
    \struc{T_{\lambda,\mu}^{\kappa}} \coloneqq (1-\kappa)\Id + \kappa P_B^{\mu}\circ P_A^{\lambda}.
\end{equation*}
\end{definition}

Note that $\mu=\lambda = \kappa = 1$ yields the alternating projection operator, and $\mu = \lambda = 2$ and $\kappa = 1/2$ generates the Douglas--Rachford operator. 

Starting from $x_0\in \R^d$, the generalized projection method applies the operator $T =T_{\lambda,\mu}^{\kappa}$ to the current iterate, generating a sequence $(x_n)_{n\in \mathbb{N}}$, with 
\[
x_{n+1} = T (x_{n}) = T^n(x_0).
\]
This method is said to \struc{\emph{converge finitely}} if for every $x_0\in \R^d$, the sequence $(T^n(x_0))_{n\in \mathbb{N}}$ converges finitely, i.e., there is an $n \in \N$ such that $T^n(x_0)\in \Fix T$, where $\Fix T$ denotes the  \struc{\emph{fixed point set}} of $T$, given by
\[\struc{\Fix T} \coloneqq \{x\in \R^d \with T(x) = x\}.\]
If we restrict the domain and codomain of $T$ to a subset $S$, we emphasize this in the set of fixed points as $\struc{\Fix_S T}$.

For the case when the two sets $A$ and $B$ are closed convex cones, we will borrow the linear algebra notation and let
\[
\struc{\Ker T} \coloneqq \{x\in \R^d \with T(x) = 0_d\},
\]
where \struc{$0_d$} is the origin of $\R^d$.

In the case when $A, B \subseteq \R^d$ are closed convex sets with $A\cap B\neq \varnothing$, it is known, e.g., from \cite[Corollary~3.9 and Theorem~3.13]{BCL04} that
\begin{equation}\label{eq:FixDR}
\Fix T_{DR} =A\cap B +(A-B)^{\circleddash} \text{~~and~~} P_A(\Fix T_{DR}) =A\cap B    
\end{equation}
and for every sequence $(x_n)_{n\in \mathbb{N}}$ generated by the Douglas--Rachford method,
\begin{equation}\label{eq:seqDR}
x_n\to x\in \Fix T_{DR} \text{~~and~~} P_A(x_n)\to P_A(x) \in A\cap B \text{~as~} n\to +\infty.    
\end{equation}
In the case of two linear subspaces, the Douglas--Rachford method also exhibits linear convergence and the rate of convergence is the cosine of the Friedrichs angle between the subspaces; see~\cite{HLN14,BBNPW14}. For the linear convergence in a more general setting, we refer to~\cite{Phan14}.

Our first result (Theorem~\ref{thm:main-result}) shows that the Douglas--Rachford method converges finitely on a pair of closed convex cones with the only exception being the case when the set of fixed points of the Douglas--Rachford operator is trivial. In particular, this means that these two sets intersect in a single point, since the set of fixed points of the Douglas--Rachford operator must contain the intersection of these sets. For an illustration of the result, see Figure~\ref{fig:main_thm}.

\begin{figure}[ht]
    \centering
    \includegraphics[width=0.5\textwidth]{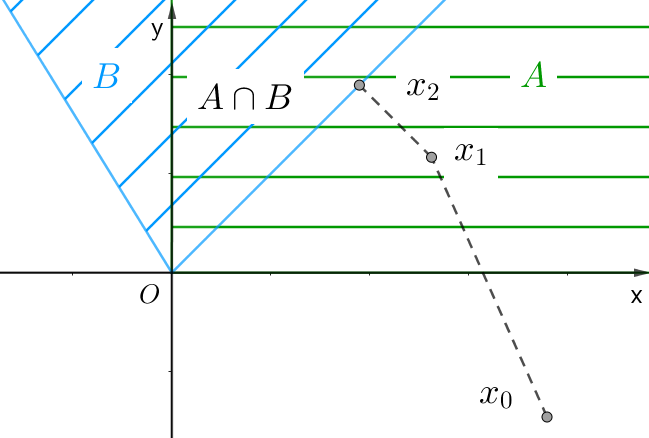}
    \hspace{0.03\textwidth}
    \includegraphics[width=0.45\textwidth]{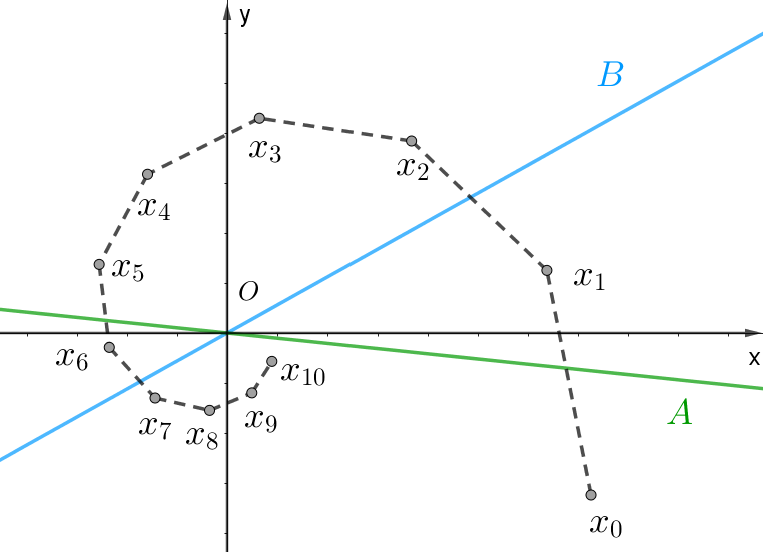}
    \caption{Left: Illustration of Theorem~\ref{thm:main-result}\ref{thm:main-result_not0}. Starting from an initial guess $x_0 \notin A \cup B$, the Douglas--Rachford method finds a point in $A \cap B$ in two steps. \\
    Right: Illustration of Theorem~\ref{thm:main-result}\ref{thm:main-result_0}. Here, $A$ and $B$ are two straight lines and $\Fix T_{DR} = A \cap B = \{0\}$. The sequence $(x_n)_{n\in \mathbb{N}}$ generated by $T_{DR}$ will converge to 0 but not finitely.}
    \label{fig:main_thm}
\end{figure}

\begin{theorem}\label{thm:main-result}
Let \(A, B \subseteq \mathbb{R}^2\) be two closed convex cones. 
\begin{enumerate}
\item\label{thm:main-result_not0}
If $\Fix T_{DR}\neq \{0\}$, then there exists an $n\in \N$ such that $T_{DR}^n(\R^2) \subseteq \Fix T_{DR}$. That is, starting from any point $x_0\in \R^2$, the Douglas--Rachford method applied to $(A, B)$ converges in at most $n$ iterations.
\item\label{thm:main-result_0}
If $\Fix T_{DR} = \{0\}$, then for any $n\in \N$, $T_{DR}^n(\R^2\setminus \Ker T_{DR})\cap \Fix T_{DR} = \varnothing$. That is, the Douglas--Rachford method applied to $(A, B)$ converges finitely if and only if the starting point $x_0$ is in the kernel of $T_{DR}$. 
\end{enumerate}
\end{theorem}
The proof of Theorem~\ref{thm:main-result} follows from Propositions~\ref{prop:Kerline} and~\ref{prop:dichotomy} given in Section~\ref{sec:finite}. We point out that finite convergence to the set of fixed points does not guarantee that the method actually produces a point in $A\cap B$, since  $\Fix T_{DR}$ may be larger than $A\cap B$. However, it is common in the literature to consider this type of convergence results; see \cite{BD17} for example. 

\begin{remark}
The condition $\Fix T_{DR} =\{0\}$ in Theorem~\ref{thm:main-result}\ref{thm:main-result_0} can be formulated in several equivalent ways. Indeed, if $A, B \subseteq \R^d$ are closed convex sets with nonempty intersection, then by \eqref{eq:FixDR},
\begin{align*}
\Fix T_{DR} =\{0\} &\iff A\cap B +(A-B)^{\circleddash} =\{0\} \\
&\iff A\cap B =\{0\} =(A-B)^{\circleddash} \\
&\iff A\cap B =\{0\} \text{~and~} 0\in \Int(A-B). 
\end{align*}
Now, assume that $A$ and $B$ are closed convex cones. Hence $A-B$ is a cone, and so $0\in \Int(A-B)$ if and only if $A-B =\R^d$. In this case,
\[
\Fix T_{DR} =\{0\} \iff A\cap B =\{0\} \text{~and~} A-B =\R^d.
\]
\end{remark}

\begin{corollary}\label{cor:cones}
Let $A,B\subseteq \R^2$ be two closed convex cones with more than one common point. Then there exists an $n\in \N$ such that for any starting point $x_0\in \R^2$, the Douglas--Rachford method applied to $(A, B)$ converges in at most $n$ iterations.   
\end{corollary}
\begin{proof}
As $A\cap B$ contains more than one point, we have from \eqref{eq:FixDR} that $\Fix T_{DR}\neq \{0\}$. The conclusion follows from Theorem~\ref{thm:main-result}\ref{thm:main-result_not0}.     
\end{proof}

We can draw the following interesting consequence for two locally polyhedral sets. Recall that a set is \struc{\emph{polyhedral at $x$}} if there exists a neighbourhood of $x$ in which the set coincides with some polyhedral set.

\begin{corollary}\label{cor:polys}
Let $A,B\subseteq \R^2$ be two closed convex sets with more than one common point such that $A$ and $B$ are polyhedral at every point in $\bdry A \cap \bdry B$ (if any). Then, for any starting point $x_0\in \R^2$, the Douglas--Rachford method applied to $(A, B)$ converges in finitely many iterations to a point in $\Fix T_{DR}$.  
\end{corollary}
\begin{proof}
%Under the assumptions of this corollary the Douglas--Rachford method starting from any point $x_0\in \R^2$ converges to some point $\bar x \in \Fix T_{DR}$. Since the behavior of a projection method is translation invariant, we can assume that $\bar x =0$. 
%
%Locally, in a sufficiently small  neighborhood of the point $\bar x$, the sets $A$ and $B$ being locally polyhedral coincide with some closed convex cones $A_0$ and $B_0$. The set of fixed points of the operator $T_0$ constructed on this pair of cones also agrees locally with the set of fixed points of the operator $T_{DR}$. Since the set of fixed points of $T_{DR}$ contains at least two points, and this set is convex, we conclude that the same is true about $T_0$. But then $\Fix T_0 \neq \{0\}$, and hence the convergence is finite.
As a set is polyhedral at each of its interior points, $A$ and $B$ are actually polyhedral at every point in $A\cap B$. Let $(x_n)_{n\in \mathbb{N}}$ be a sequence generated by the Douglas--Rachford method applied to $(A, B)$. According to \eqref{eq:seqDR}, $(x_n)_{n\in \mathbb{N}}$ converges to a point $x\in \Fix T_{DR}$ with $P_A(x)\in A\cap B$. By the local polyhedrality of $A$ and $B$, there exists a neighbourhood of $P_A(x)$ in which $A$ and $B$ coincide with some polyhedral sets. Shrinking the neighbourhood and translating the sets if necessary, $A$ and $B$ coincide with some closed convex cones $A_0$ and $B_0$ in a neighbourhood of $P_A(x)$. Since $A\cap B$ contains more than one point, so does $A_0\cap B_0$. In view of \cite[Theorem~5.2]{BD17} and Corollary~\ref{cor:cones}, $(x_n)_{n\in \mathbb{N}}$ converges finitely to $x$.
\end{proof}

It can be seen that Corollary~\ref{cor:polys} covers the result in $\R^2$ of \cite[Theorem~4.5]{BD17}. We also note that if $A$ is convex and $A\cap \Int B\neq \varnothing$, then $A\cap B$ contains more than one point. Therefore, the result in $\R^2$ of \cite[Theorem~3.7]{BDNP16} follows from Corollary~\ref{cor:polys}.

\medskip

Our second result (Theorem~\ref{thm:badmethods})  establishes that the Douglas--Rachford method is the only one in the family of projection methods having the strong finite convergence behavior given in Theorem~\ref{thm:main-result}, hence justifying  referring to it as the best projection method.

\begin{theorem}\label{thm:badmethods}  
For any allowable choice of the parameters $\mu, \lambda, \kappa $ in Definition~\ref{def:projop}, except for $(\mu, \lambda, \kappa) = (2,2, 1/2)$, namely for all $T = T_{\lambda,\mu}^{\kappa} \neq T_{DR}$, there exists a pair of closed convex cones $A,B\subseteq \R^2$ and a starting point $x_0\in \R^2$ such that, for any $n\in \N$, $T^n(x_0) \notin \Fix T\neq \{0\}$. 
\end{theorem}
We prove Theorem~\ref{thm:badmethods} in Section~\ref{sec:counterexamples} by constructing explicit examples for each choice of the parameters.

\section{Finite convergence of the Douglas--Rachford method}\label{sec:finite}

In this section, we prove Theorem~\ref{thm:main-result}. We do so by first considering the case when the kernel of the Douglas--Rachford operator is a line (Section~\ref{sec:linecase}) and then reducing the remaining cases to circle dynamics (Section~\ref{sec:circle}).

\subsection{The case when the kernel is a line}\label{sec:linecase}

Our goal is to prove Proposition~\ref{prop:Kerline}, which yields Theorem~\ref{thm:main-result} in the case when the kernel of the Douglas--Rachford operator is a line. 
Using the polar, we deduce the following convenient characterization of the kernel of the Douglas--Rachford operator.

\begin{lemma}\label{lem:kernel} 
Let $A,B\subseteq \R^d$ be closed convex cones. Then
\begin{equation}\label{eq:DRkernelRepresentation}
\Ker T_{DR} = \cone\,[(-B \cap A^{\circleddash}) \cup (B^{\circleddash} \cap A)].
\end{equation}
\end{lemma}

\begin{proof} 
We first observe that the set $-B \cap A^{\circleddash}$ is orthogonal to $B^{\circleddash} \cap A$, i.e.,
\begin{equation}\label{eq:orthogonaltechnical}
\text{for all~} u \in -B \cap A^{\circleddash} \text{~and all~} v\in B^{\circleddash} \cap A,\quad \langle u,v\rangle = 0.    
\end{equation}
Indeed, take any $ u \in -B \cap A^{\circleddash}$ and $v\in B^{\circleddash} \cap A$. Since $u\in A^{\circleddash} $ and $v\in A$,  we have $\langle u,v \rangle \leq 0$. However, as $u\in -B$ and $v\in B^{\circleddash}$ we also have $\langle u,v \rangle \geq 0$, and so \eqref{eq:orthogonaltechnical} follows.

For convenience, denote  
\[
Q \coloneqq \cone [(-B \cap A^{\circleddash}) \cup (B^{\circleddash} \cap A)].
\]
We show that $Q\subseteq \Ker T_{DR}$. Consider an arbitrary $x\in Q$. Then $x = u+v$ for some $u \in -B \cap A^{\circleddash}$ and $v\in B^{\circleddash} \cap A$ (by virtue of these two sets being closed convex cones). Since $\langle u,v\rangle = 0$ due to \eqref{eq:orthogonaltechnical}, we derive that, for all $w\in A$, 
\begin{equation*}\label{eq:technicalprojxv}
\langle x - v, w-v\rangle = \langle u, w-v\rangle = \langle u, w \rangle \leq 0,
\end{equation*}
because $u\in A^{\circleddash}$. Since $v\in A$, by the angle characterization of projection, we conclude that $v$ is the projection of $x$ onto $A$. That is, $P_A(x) = v$, and thus, $R_A(x) = 2 v - x = 2v - (u+v) = v-u$. 
Likewise, for all $w\in B$, 
\[
\langle (v-u) - (-u), w-(-u)\rangle = \langle v, w+u\rangle = \langle v,w\rangle \leq 0, 
\]
where the final inequality is due to $v\in B^{\circleddash}$. We hence deduce that $-u$ is the projection of $v-u$ onto $B$, i.e.,  $P_B(v-u) = -u$. This yields 
\[
R_B \circ R_A (x) = R_B(v-u) = 2 (-u) - (v-u) = -u-v = -x,
\]
and therefore,
\[
T_{DR}(x) = \frac{1}{2}\left(x+R_B \circ R_A (x)\right) = \frac{1}{2}(x-x) = 0,
\]
which means $x\in \Ker T_{DR}$. 

It remains to show that $\Ker T_{DR} \subseteq Q$. Consider $x\in \Ker T_{DR}$. Then 
\[
0 = T_{DR}(x) = \frac{1}{2}\left (x+ R_B \circ R_A (x)\right),
\]
and thus $R_B \circ R_A (x) = -x$. 
Now let $y= R_A(x)$. Then $P_A(x) = \frac 1 2 (x+y)$ and because $R_B(y) = -x$, also $P_B(y) = \frac 1 2 (y-x)$. It follows that $P_B(y) = -x +P_A(x)$ and $P_A(x) = y -P_B(y)$.

Since $A$ is a cone, for all $\lambda \in (0, +\infty$, we have $\lambda P_A(x) \in A$, and hence by the characterization of projection,
\[
0\geq \langle x-P_A(x) , \lambda P_A(x) - P_A(x)\rangle = (\lambda -1) \langle x-P_A(x) ,P_A(x)\rangle, 
\]
and therefore $\langle x-P_A(x) ,P_A(x)\rangle = 0$. This leads to
\begin{equation}\label{eq:24532352}
\langle P_B(y), P_A(x)\rangle  = -\langle x- P_A(x),P_A(x)\rangle = 0.    
\end{equation}
The angle characterization of projection yields
\begin{equation}\label{eq:24532353}
\text{for all } u\in A,\quad \langle x- P_A(x),u-P_A(x)\rangle \leq 0,
\end{equation}
and 
\begin{equation}\label{eq:24532354}
\text{for all } v\in B,\quad \langle y-P_B(y),v-P_B(y)\rangle \leq 0.
\end{equation}
Now, for all $v\in B$, using~\eqref{eq:24532352},~\eqref{eq:24532353}, and the fact that $P_A(x) = y -P_B(y)$, we have
\[
\langle P_A(x), v\rangle = \langle P_A(x), v\rangle - \langle P_A(x), P_B(y)\rangle  = \langle P_A(x), v- P_B(y)\rangle  = \langle y-P_B(y), v- P_B(y)\rangle \leq 0,
\]
therefore $P_A(x)\in B^\circleddash$.
Analogously, for all $u\in A$, using~\eqref{eq:24532352},~\eqref{eq:24532354}, and the fact that $P_B(y) = -x +P_A(x)$ leads to
\[
\langle P_B(y),u\rangle = \langle P_B(y),u\rangle - \langle  P_B(y), P_A(x)\rangle  = \langle  P_B(y), u 
- P_A(x)\rangle = -\langle  x-P_A(x),
u - P_A(x)\rangle \geq 0,
\]
and thus $- P_B(y) \in A^\circleddash $. 

We conclude 
\[
x = P_A(x)+ (-P_B(y)),
\]
where $P_A(x) \in A\cap B^\circleddash$ and $-P_B(y)\in -B\cap A^\circleddash $, which means that $x\in Q$.
\end{proof}

%\begin{remark}
Note that the proof of Lemma~\ref{lem:kernel} shows that the two sets featuring in the representation of the kernel are orthogonal. It may happen that one of them is trivial, but it is perfectly possible that both of them are nontrivial, for instance, when $A$ and $B$ are two lines orthogonal to each other. These observations may be useful when considering the case of the kernel being a line.
%\end{remark}

\begin{proposition}\label{prop:Kerline}
Let $T_{DR}\colon\R^2\to \R^2$ be the Douglas--Rachford operator generated by two closed convex cones $A,B\subseteq \R^2$. Suppose that $\Ker T_{DR}$ is a line. Then %the process converges to $\Fix T_{DR}$ in at most one step.
$\Fix T_{DR}\neq \{0\}$ and $T_{DR}(\R^2) \subseteq \Fix T_{DR}$.
\end{proposition}
\begin{proof}
By Lemma~\ref{lem:kernel},  
\[
\Ker T_{DR} = \cone\,[(-B \cap A^{\circleddash}) \cup (B^{\circleddash} \cap A)].
\]
Since $-B \cap A^{\circleddash}$ and  $B^{\circleddash} \cap A$ are orthogonal, if $\Ker T_{DR}$ is a line, either $-B \cap A^{\circleddash} $ is a line and $B^{\circleddash} \cap A = \{0\}$, or the other way around.

\emph{Case 1:} If \(-B \cap A^{\circleddash}=L\), where $L$ is a line through zero, and \(B^{\circleddash} \cap A\) is \(\{0\}\), then $A,B^{\circleddash}\subseteq L^\perp$. 
    \begin{itemize}
        \item If $A=\{0\}$, then $A^{\circleddash} = \R^2$ and $B = L$;
        \item If $B^{\circleddash}=\{0\}$, then $B=\R^2$ and $A^{\circleddash} = L$;
        \item If $A\neq \{0\}$ and $B^{\circleddash}\neq \{0\}$, we must have $A=-B^{\circleddash}$, coinciding with a half-line orthogonal to $L$, but this is impossible since in this case $-B\cap A^{\circleddash}$ would be a half-space.
    \end{itemize}
It follows that either $A = \{0\}$ and $B=L$ or $A = L^\perp$ and $B=\R^2$.

\emph{Case 2:} If \(B^{\circleddash} \cap A=L\) is a line and \(-B \cap A^{\circleddash}\) is \(\{0\}\), analogously (or by letting $\tilde A = B$, $\tilde B = -A$ and applying the previous case) we have 
    \begin{itemize}
        \item If $B=\{0\}$, then $B^{\circleddash} = \R^2$ and $A = L$;
        \item If $A^{\circleddash}=\{0\}$, then $A=\R^2$ and $B^{\circleddash}= L$;
        \item If $B\neq \{0\}$ and $A^{\circleddash}\neq \{0\}$, then $B^{\circleddash}\cap A$ is a half-space, contradicting the assumption.
    \end{itemize}
    Hence either $A = \R^2$ and $B=L^\perp$ or $B=\{0\}$ and $A = L$. 

For all four possible cases, it is trivial to check that \(T_{DR}(\R^2) \subseteq \Fix T_{DR}\). Moreover, since $\Fix T_{DR} = (A\cap B) +(A-B)^{\circleddash}$ by \eqref{eq:FixDR}, we deduce that $\Fix T_{DR}\neq \{0\}$.
\end{proof}

% \begin{proposition}\label{prop:KerLineFixNontriv} If the  kernel of the Douglas--Rachford operator $T_{DR}$ generated by two closed convex cones $A,B\subseteq \R^2$ is a line, then $\Fix T_{DR} \neq\{0\}$. 
% \end{proposition}
% \begin{proof}
% The proof of Proposition~\ref{prop:Kerline} shows that if the kernel is a line, then one of the sets is a line and the other one is either zero or the whole space. Together with the representation of the fixed points
%     \[
%     \Fix T_{DR} = (A\cap B) + (A-B)^{\circleddash},
%     \]
%     see \cite[Corollary~3.9]{BCL04}, this means that the set of fixed points is never trivial. 
% \end{proof}

\subsection{Reduction to circle dynamics when the kernel is not a line}\label{sec:circle}

For the case when the kernel of the Douglas--Rachford operator is not a line, the proof of Theorem~\ref{thm:main-result} can be reduced to the study of circle dynamics of an associated system. We first explain this identification and derive useful results that allow us to proceed with the proof of Proposition~\ref{prop:dichotomy} (and together with Proposition~\ref{prop:Kerline} %and Proposition~\ref{prop:KerLineFixNontriv} 
proves Theorem~\ref{thm:main-result}, as we explain at the end of this section).

Note that if $r$ is a ray, then either it is trivial ($r=\{0\}$), or it has a unique intersection with the unit circle at $x/\|x\|$. Likewise, every point on the circle defines a unique ray. We can hence identify points on the unit circle with nontrivial rays.  Denote the \struc{\emph{ray generated by}} $\struc{x}\in \R^2$ by \struc{$[x]$}, i.e., $[x] = \{\lambda x \with \lambda \in (0, +\infty)\}$. Now define the mapping \struc{$\arg$} that maps nontrivial rays to the point on the unit circle \struc{$S^1$},  
\[
\text{for all } x \neq 0,\quad \arg([x]) = x/\|x\|.
\]
Since this mapping is bijective, its inverse is well defined:
\[
\text{for all } t\in S^1,\quad \arg^{-1}(t) = [t].
\]

It turns out that the Douglas--Rachford operator generated by two closed convex cones maps rays to rays, i.e., it is  \struc{\emph{positively homogeneous}}. Moreover, it can only map rays to the trivial ray `once'. Thus the behaviour of the method can be reduced to circle dynamics. We make this more precise in the following results. 

\begin{lemma}\label{lem:homogeneous}
Let $T_{\lambda,\mu}^{\kappa}\colon \R^d\to \R^d$ be the generalized projection operator generated by two closed convex cones $A, B \subseteq \mathbb{R}^d$. Then $T_{\lambda,\mu}^{\kappa}$ is positively homogeneous, that is, for all $x\in \R^d$ and all $\lambda \in (0, +\infty)$,
\[
T_{\lambda,\mu}^{\kappa}(\lambda x) = \lambda T_{\lambda,\mu}^{\kappa}(x).
\]
Consequently, the Douglas--Rachford operator $T_{DR}$ generated by two closed convex cones \(A, B \subseteq \mathbb{R}^d\) is positively homogeneous.
\end{lemma}
\begin{proof}
This follows straightforwardly from the definition of the generalized projection operator. Since the identity mapping and projections onto cones are positively homogenenous, so are their compositions and linear combinations. 
\end{proof}

Lemma~\ref{lem:homogeneous} allows us to introduce a well-defined circle map induced by the Douglas--Rachford operator. 
More generally, suppose that $T\colon\R^2\to \R^2$ is a positively homogeneous mapping. Then $T$ maps rays to rays, that is,  $T([x]) = [T(x)]$. Therefore, we can define the associated \struc{\emph{unit circle operator}} $\struc{\varphi_T}\colon S^1\to S^1\cup \{0\}$, 
\begin{equation}\label{eq:assocoperator}
\text{for all } t \in S^1,\quad \varphi_T(t) \coloneqq \arg (T([t])) = \arg ([T(t)]).
\end{equation}
It is evident that this mapping is well defined, moreover taking the inverses we have 
\begin{equation}\label{eq:assocoperator2}
\arg^{-1} \varphi_T(t) = [T(t)] = T([t]) 
\end{equation}
whenever $T(t)\neq 0$. 

Having zero in the range of $\varphi_T$ prevents us from working exclusively on the circle. Fortunately, we can suitably restrict this mapping to a continuous subset of the circle and hence avoid dealing with this pesky zero. A key observation is that the Douglas--Rachford operator never maps to nonzero points in its kernel, and hence any starting point is either mapped to zero in the first iteration or cannot be mapped to zero in finitely many iterations.

The next result allows us to pinpoint the specific linear pieces that comprise the Douglas--Rachford operator. We note that this is the key result that distinguishes the behavior of this method from the rest of the family $T_{\lambda,\mu}^{\kappa}$, as any other combination of parameters generate different types of linear pieces, that, in particular, prevent us from identifying the set of fixed points of the circle map with the set of fixed points of the generalised projection method operator. 

\begin{proposition}\label{prop:kerintfix}
Let $T_{DR}\colon\R^d\to \R^d$ be the Douglas--Rachford operator generated by two closed convex cones $A,B\subseteq \R^d$. Then 
\[
T_{DR}(\R^d) \cap \Ker T_{DR} = \{0\}.
\]
\end{proposition}
\begin{proof} 
Suppose the statement is not true, then there exists $y\in T_{DR}(\R^d)\cap \Ker T_{DR}$ such that $y\neq 0$. Since $y\in T_{DR}(\R^d)$ and $y\neq 0$, there must be some $x\in \R^d$ with $x\notin \Ker T_{DR}$ such that $T_{DR}(x) = y$. 
%By Lemma~\ref{lem:kernel}, we know that  \(\Ker T_{DR}\) is a convex cone.
%the kernel of the Douglas--Rachford operator is a convex set. It also must be a cone since the operator is positively homogeneous. We conclude that \(\Ker T\) is a convex cone.
Let $\lambda \in (0, +\infty)$. Since $T_{DR}$ is firmly nonexpansive~\cite[Proposition~3.1]{BCL04}, we have from \cite[Proposition~4.4(i)\&(iv)]{BC17} that
\begin{align*}
\|T_{DR}(x)-T_{DR}(\lambda y)\|^2 \leq \langle x-\lambda y, T_{DR}(x) - T_{DR}(\lambda y)\rangle.
\end{align*}
Since $y = T_{DR}(x) \in \Ker T_{DR}$, by Lemma~\ref{lem:homogeneous} we have \(T_{DR}(\lambda y) = \lambda T_{DR}(y) =0\). Therefore,
\begin{align*}
    \|T_{DR}(x)\|^2 \leq \langle x-\lambda T_{DR}(x), T_{DR}(x)\rangle,
\end{align*}
which yields
\begin{align*}
(1+\lambda)\|T_{DR}(x)\|^2 \leq \langle x,  T_{DR}(x) \rangle.
\end{align*}
Since $T_{DR}(x) = y \neq 0$, letting $\lambda$ go to positive infinity, the left hand side approaches positive infinity, which is impossible, since the right hand side \(\langle T_{DR}(x),x \rangle\) is a constant. We reached a contradiction, hence our assumption is wrong and the original statement is true. 
\end{proof}

\begin{proposition}\label{prop:mappingtypes} 
Let $T_{DR}\colon\R^2\to \R^2$ be the Douglas--Rachford operator generated by two closed convex cones $A,B\subseteq \R^2$. Then $T_{DR}$ is a (continuous) piecewise linear mapping such that each piece is one of the following types of a linear mapping:
\begin{enumerate}
    \item a composition of a rotation by some fixed angle $\theta$ and a scaling by $|\cos \theta|$;
    \item a projection onto some fixed line; 
    \item a projection onto zero; or
    \item the identity mapping.
\end{enumerate}
\end{proposition}
\begin{proof}
Depending on the point $x\in \R^2$ the projection $P_C$ of $x$ onto a closed convex cone $C\subseteq \R^2$ is either a reflection in zero (i.e., $P_C(x) = -x$), a reflection in a line (that corresponds to one of the (at most two) edges of the cone $C$), or the identity mapping. Therefore, $P_C(x)$ is a piecewise linear operator with parts being one of these three possibilities. Hence, for every point $x\in \R^2$,  the Douglas--Rachford operator is a composition of a finite number of such elementary linear mappings, and thus is in itself a piecewise linear mapping. Based on whether for a particular point $x$ the reflections $R_A(x)$ and $R_B(R_A(x))$ are reflections in zero, in a line, or are identity mappings, we have the nine cases listed in Table~\ref{table:T}.
Note that if both \(R_A,R_B\) are reflections in lines $L_1$ and $L_2$ respectively, suppose the angle between $L_1$ and $L_2$ is $\theta$, then \(|\arg (x)-\arg R_B(R_A(x))|=2\theta\) so that \(\|T_{DR}(x)\|=|{\rm cos}\, \theta|\|x\|\).

\begin{table}[!ht]
\begin{center}
\begin{tabular}{p{8em} | p{9em} p{7em} p{7em}}
% \toprule
\diag{0.1em}{8em}{\;\;$R_B$}{$R_A$\;\;\hfill} & \multicolumn{1}{c}{Reflect in a line $L_1$} & \multicolumn{1}{c}{Reflect in $0$} & \multicolumn{1}{c}{Identity} \\
\hline
Reflect in a line $L_2$ & Rotate by an angle $\theta$ and scale by \(|\cos  \theta|\) & Project onto $L_2^\perp$ & Project onto $L_2$ \\

Reflect in $0$ & Project onto \(L_1^{\perp}\) & Identity & Project onto 0 \\

Identity & Project onto \(L_1\) & Project onto 0 & Identity %\\
% \bottomrule
\end{tabular}
\end{center}
\caption{Possible linear pieces of the mapping $T_{DR}$.}
\label{table:T}
\end{table}
\end{proof}

\begin{proposition}\label{prop:typesphi} 
Suppose that $I$ is a connected subset of the unit circle and $\varphi\colon I \to I$ is a continuous piecewise function, such that each of the pieces is either an identity, maps to some particular point, or is a fixed rotation. 
If $\Fix \varphi \neq \varnothing$, then there exists an $n\in \mathbb{N}$ such that $\varphi^n (I)\subseteq \Fix \varphi$. 
\end{proposition}
\begin{proof}
First, note that $\varphi$ is nonexpansive with respect to arc lengths. Indeed, since $I$ is connected, for any $x,y\in I$, at least one of the two arcs connecting them lies in $I$. By \struc{$d(x,y)$} denote the length of the shortest of these arcs (in case there are two). We claim that the mapping $\varphi$ is nonexpansive with respect to this metric. Assume the contrary, then there exist some $x,y\in I$ such that the nonexpansiveness is broken with some $\alpha>1$,
that is,
\[
d(\varphi(x), \varphi(y)) \geq \alpha d(x,y).
\]
On the shortest arc connecting $x$ and $y$, choose the midpoint $z$. Then at least one of the two resulting segments will also fail the nonexpansiveness with the same value $\alpha$. We can continue the subdivisions until we generate a converging sequence of nested segments $[x_n, y_n]$ that converges to some $\bar x$. Since $\bar x\in [x_n,y_n]$, either $[x_n,\bar x]$ or $[\bar x,y_n]$ breaks nonexpansiveness. Choosing the expansiveness failing interval and relabeling, we have $x_n\to \bar x$, and $\varphi $ breaking nonexpansiveness on each one of these segments with the same $\alpha$. Because there are finitely many pieces, we can refine the sequence further to ensure that $x_n$ is always mapped by the same piece. By continuity, the value $\varphi(\bar x)$ must be consistent. This piece is nonexpansive, since it is either a rotation, a constant map or the identity map, so we have a contradiction and thus, $\varphi$ is nonexpansive. 

By continuity of the mapping $\varphi$, there must be a sufficiently small neighborhood of $\Fix \varphi$ in $I$ where $\varphi$ is the projection onto some point in $\Fix \varphi$. Indeed, if this is not the case, then there exists a sequence $(x_n)_{n\in \mathbb{N}}$  of points in $I \setminus \Fix \varphi $ such that $\dist (x_n, \Fix \varphi)$ approaches zero, and $\varphi (x)$ is identified with a piece that is not a projection onto $\Fix \varphi$. Since the circle is compact, we can assume that $(x_n)_{n\in \mathbb{N}}$ converges to some point $\bar x \in \Fix \varphi$, and because the number of pieces is finite, we can choose a subsequence of $(x_n)_{n\in \mathbb{N}}$ such that $\varphi$ acts on $x$ using the same piece. If this piece is a fixed rotation, then $\varphi(\bar x) \neq \bar x$, which is impossible. If the piece is a projection onto some point, then this point must be $\bar x$, which contradicts our assumption.
We conclude that there exists a sufficiently small neighborhood of $\Fix \varphi$, say, of size $\varepsilon$ such that this entire neighborhood is mapped to the set of fixed points. Then for any point $x\in I$ we have either $\varphi(x)\in \Fix \varphi$ or we connect $x$ to $\Fix \varphi$ with the shortest arc (this is possible since the set of fixed points is closed in $I$) and take the point $y$ on this arc at the distance $\varepsilon$ of the endpoint. Then we have 
\[
d(\varphi (x), \Fix \varphi)  \leq d (\varphi (x), \varphi(y)) \leq d(x,y)  =  d(x,\Fix \varphi) - \varepsilon.
\]
Since with each application we either bring the point to the set of fixed points or reduce the distance to the set of fixed points by $\varepsilon$, after at most $n= 2\pi/\varepsilon$ iterations we have $\varphi^n(I) \subseteq \Fix \varphi$.   
\end{proof}

\begin{proposition}\label{prop:dichotomy} 
Let $T_{DR}\colon\R^2\to \R^2$ be the Douglas--Rachford operator generated by two closed convex cones $A,B\subseteq \R^2$. Suppose that $\Ker T_{DR}$ is not a line. Then exactly one of the following alternatives holds.
\begin{enumerate}
\item\label{prop:dichotomy_not0}
If $\Fix T_{DR} \neq  \{0\}$, then there is an $n\in \N$ such that $T_{DR}^n(\R^2) \subseteq \Fix T_{DR}$.
\item\label{prop:dichotomy_0}
If $\Fix T_{DR} =  \{0\}$, then for any $n\in \N$, $T_{DR}^n(\R^2\setminus \Ker T_{DR}) \cap \Fix T_{DR} = \varnothing$.
\end{enumerate}
\end{proposition}
\begin{proof} 
Let $W = \R^2 \setminus \Ker T_{DR}$. Observe that $0\notin W$ and that, by Proposition~\ref{prop:kerintfix}, $T_{DR}(W)\subseteq W$. 

To show~\ref{prop:dichotomy_0}, let $\Fix T_{DR} = \{0\}$. Then $\Fix T_{DR}\cap W = \varnothing$, and given that $T_{DR}^n(W) \subseteq W$, we conclude that $T_{DR}^n(W) \cap \Fix T_{DR} = \varnothing$ for all $n$. 

For~\ref{prop:dichotomy_not0}, assume $\Fix T_{DR}\neq \{0\}$. 
Recall that by Lemma~\ref{lem:kernel}, $\Ker T_{DR}$ is a convex cone. Hence its complement $W$ is also a cone. Moreover, since $\Ker T_{DR}$ is a convex cone but not a line, its intersection with the unit circle is connected. The set $I = W\cap S^1$ is therefore connected. Thus, restricting domain and codomain to $I$, the mapping $\varphi_T$ in~\eqref{eq:assocoperator} is well defined. 
Taking into account that $0 \notin T_{DR}(W)$, Lemma~\ref{lem:homogeneous} leads firstly to $T_{DR}([x])= [x]$ being equivalent to $T_{DR}(x) = x$, that is, $T_{DR}$ maps a ray to itself if and only if it fixes every point of the ray. This means that restricting $T_{DR}$ to $W$ and $\varphi_{T_{DR}}$ to $I $ we have 
\begin{equation}\label{eq:fixcoincide}
\arg([\Fix_W T_{DR}]) = \Fix_I \varphi_{T_{DR}}, \;\text{ or equivalently }  [\Fix_W T_{DR}] = \arg^{-1}(\Fix_I \varphi_{T_{DR}}),
\end{equation}
where by \struc{$[S]$} we denote the set of all rays in a cone $S$. 

Secondly, Proposition~\ref{prop:mappingtypes} yields that restricted to the rays in $[W]$, the mapping $T_{DR}([x])$ is piecewise linear with pieces being identity, rotation, or a mapping onto a (different) fixed ray. 
It is easy to see that since $T_{DR}$ is nonexpansive, and hence continuous on $W$, the associated $\varphi_{T_{DR}}$ must be continuous on~$I$. 
From~\eqref{eq:fixcoincide} we conclude that $\Fix_I \varphi_{T_{DR}} \neq \varnothing$. It follows from Proposition~\ref{prop:typesphi} that there exists some $n\in \N$ such that $\varphi_{T_{DR}}(I)\subseteq \Fix_I \varphi_{T_{DR}}$. Then for any $[x]\subseteq [W]$,
\begin{align*}
T_{DR}^n([x])  \overset{\eqref{eq:assocoperator2}}{=}  T_{DR}^{n-1}(\arg^{-1} \varphi_{T_{DR}}(I)) \overset{\eqref{eq:assocoperator2}}{=}  T_{DR}^{n-2}(\arg^{-1} \varphi_{T_{DR}}^2(I)) = \cdots = \arg^{-1} \varphi_{T_{DR}}^n(I).  
\end{align*}
Now note that 
\begin{align*}
\arg^{-1} \varphi_{T_{DR}}^n(I) \overset{\text{Prop.}~\ref{prop:typesphi}}{\subseteq} \arg^{-1} \Fix_I \varphi_{T_{DR}} \overset{\eqref{eq:fixcoincide}}{=}  [\Fix_W T_{DR}].
\end{align*}
We conclude, $T_{DR}^n (W) = T_{DR}^n(\R^2\setminus \Ker T_{DR})\subseteq \Fix T_{DR}$ which together with $T_{DR}(\Ker T_{DR}) = \{0\}\subseteq \Fix T_{DR}$ (by Proposition~\ref{prop:kerintfix}) proves the case~\ref{prop:dichotomy_not0}.
\end{proof}

Combining the results of Section~\ref{sec:finite}  finally shows Theorem~\ref{thm:main-result}. 

\begin{proof}[Proof of Theorem~\ref{thm:main-result}] If $\Ker T_{DR}$ is not a line the statement follows from Proposition~\ref{prop:dichotomy}. In the case when the kernel is a line, Proposition~\ref{prop:Kerline} shows that the set of fixed points is nontrivial and %finite convergence follows from Proposition~\ref{prop:Kerline}.
the Douglas--Rachford method converges in at most one step.
\end{proof}

\section{Counterexamples}\label{sec:counterexamples}

In this section, we prove Theorem~\ref{thm:badmethods} by constructing explicit examples for each choice of the three parameters, where the method does not converge finitely, even then the operator has a nontrivial set of fixed points. We present six different (families of) counterexamples, each covering a subset of the parameter space from the Definition~\ref{def:projop}. Note that $A \cap B \subseteq \Fix T_{\lambda,\mu}^{\kappa}$ for all possible choices of $(\lambda,\mu,\kappa)$. In the examples below, we choose $A \cap B$ to be nontrivial so that each $\Fix T_{\lambda,\mu}^{\kappa}$ is nontrivial as well. The union of these subsets together with the choice of parameters that corresponds to the Douglas--Rachford method covers the whole parameter space. 

More specifically, in each example we demonstrate that for the prescribed choice of parameters $(\lambda,\mu,\kappa)$ there exists a pair of sets $A$ and $B$ in the plane and a starting point $x_0$ such that the set of fixed points of the corresponding operator $T_{\lambda,\mu}^{\kappa}$ is nonempty, while the projection method based on $(\lambda,\mu,\kappa)$ and the pair $A,B$, applied to $x_0$, does not converge finitely. 
Examples~\ref{Ex1},~\ref{Ex2}, and~\ref{Ex3} cover all parameter combinations except for the curve $(1/t,1/t,t)$, $t\in [1/2,+\infty)$, see Figure~\ref{fig:examples1}. 
\begin{figure}[ht]
    \centering
    \includegraphics[width=0.3\textwidth]{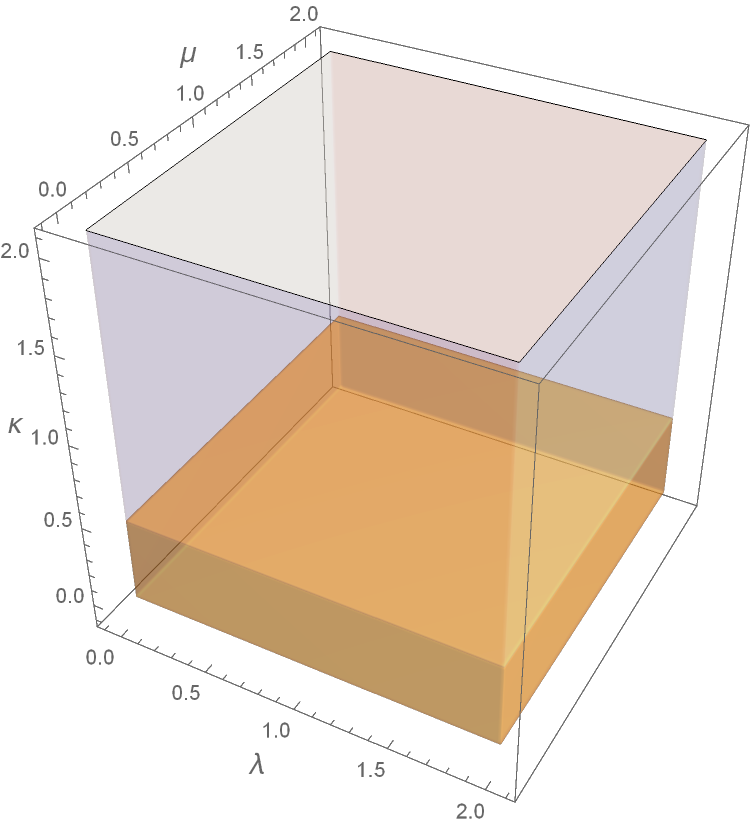}
    \includegraphics[width=0.3\textwidth]{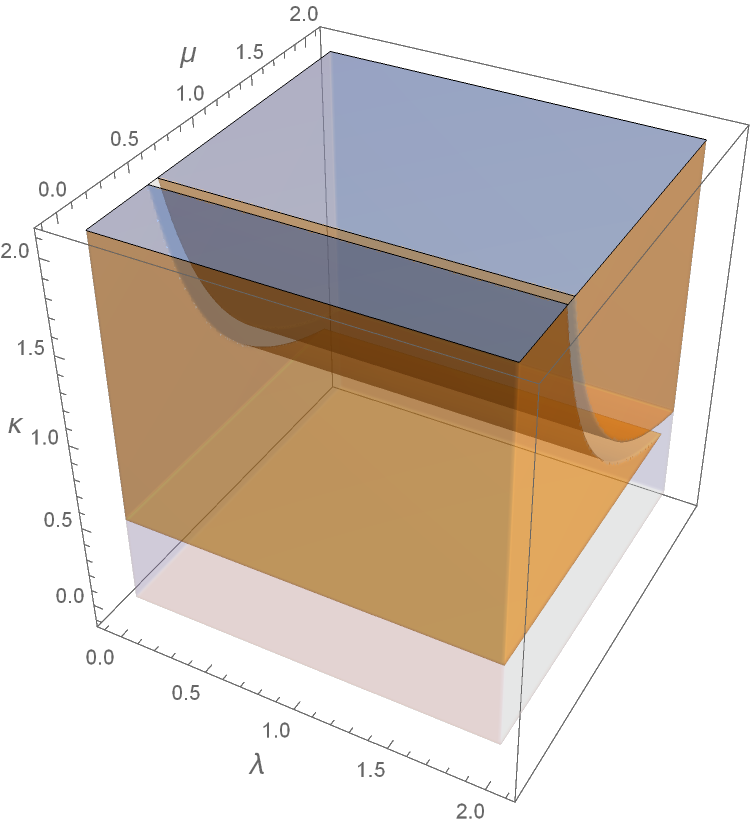}
    \includegraphics[width=0.3\textwidth]{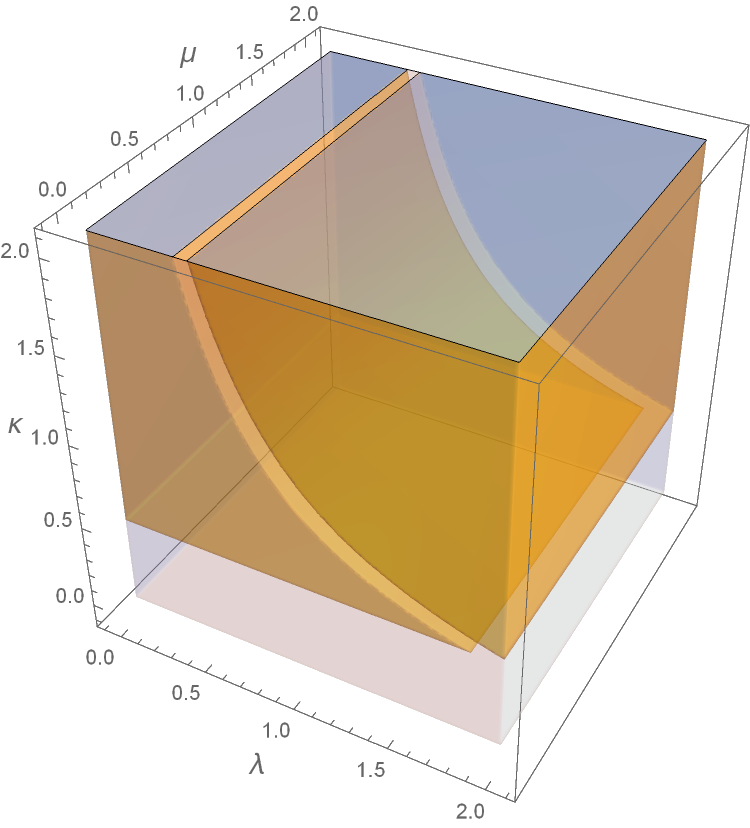}
    \caption{From left to right: the values of parameters $(\lambda, \mu,\kappa)$ covered by Examples~\ref{Ex1}, \ref{Ex2}, and \ref{Ex3}.}
    \label{fig:examples1}
\end{figure}
Examples~\ref{Ex4},~\ref{Ex5}, and~\ref{Ex6} address the values $t\in (1/2,+\infty)$, and the only remaining choice $(2,2,1/2)$ corresponds to the Douglas--Rachford method, see Figure~\ref{fig:examples2}.
\begin{figure}[ht]
    \centering
    \begin{overpic}[%grid, 
    width=0.55\textwidth]{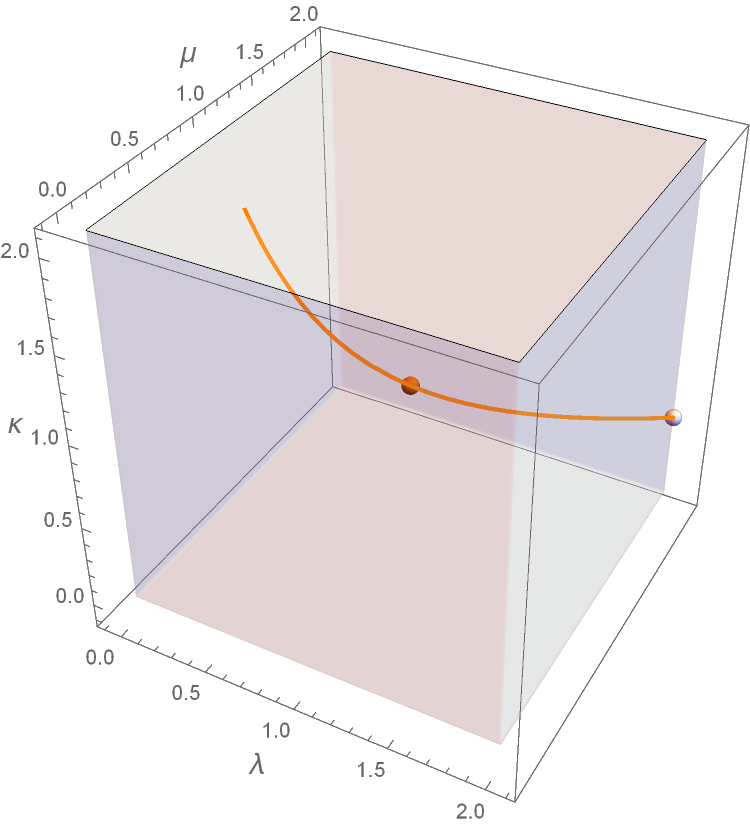}
\put(35,72){{Example~\ref{Ex5}}}
\put(25,51){{Example~\ref{Ex6}}}
\put(55,44){Example~\ref{Ex4}}
\put(69,55){Douglas--Rachford}
\end{overpic}
    \caption{The values of parameters $(\lambda, \mu,\kappa)$ covered by Examples~\ref{Ex4}, \ref{Ex5}, and \ref{Ex6}, and the single choice corresponding to the Douglas--Rachford method.}
    \label{fig:examples2}
\end{figure}

For convenience of describing the concrete subsets of the plane, we mostly use polar coordinates.

\begin{figure}[!ht]
    \centering
    \includegraphics[width = 0.5\textwidth]{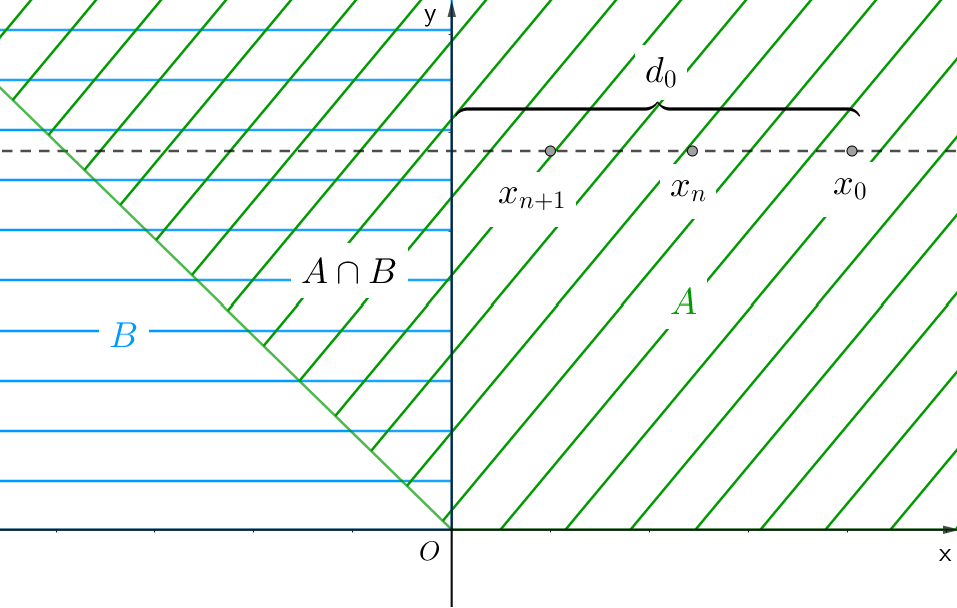}
    \caption{Illustration of Example~\ref{Ex1}, when $(\lambda,\mu,\kappa) \in (0,2] \times (0,2] \times (0,1/2)$.}
    \label{fig:Ex1}
\end{figure}

\begin{example}\label{Ex1}
Let $A = \{(\rho , \theta) \with 0 \leq \theta \leq 3\pi/4\}$, $B = \{(\rho , \theta)\with \pi/2 \leq \theta \leq \pi\}$, and set \(x_0 \in \Int A\). Then for any $n \in \mathbb{N}$, $T^n(x_0) \notin \Fix T$ with $T =T_{\lambda,\mu}^{\kappa}$ and
\[(\lambda,\mu,\kappa) \in (0,2] \times (0,2] \times (0,1/2).\]
\end{example}
\begin{proof}
We know $A\cap B = \{(\rho ,\theta)\with \pi/2 \leq \theta \leq 3\pi/4\}$. Let $L\coloneqq\{(\rho , \theta)\with \theta = \pi/2\}$. Suppose the Cartesian coordinates of $x_0$ are $x_0 = (d_0,y)$, where $d_0>0,y>0$, thus $d(x_0,L)=d_0$, as illustrated in Figure~\ref{fig:Ex1}.
Since $P_A^{\lambda}(x_0) = x_0$, we have
\begin{align}\label{eq:Ex1_1}
P_B^{\mu}(P_A^{\lambda}(x_0))=P_B^{\mu}(x_0)=(1-\mu) \cdot (d_0,y)+\mu\cdot (0,y) = ((1-\mu) d_0,y),
\end{align}
and therefore
\begin{align}\label{eq:Ex1_2}
T(x_0)=(1-\kappa)x_0 + \kappa P_B^{\mu}( P_A^{\lambda}(x_0)) = (1-\kappa) (d_0,y)+ \kappa ((1-\mu) d_0,y) = ((1-\kappa \mu) d_0,y).
\end{align}
Similarly, we can claim 
\[T^n(x_0) = ((1-\kappa \mu)^n d_0,y).\]
Given $(\mu,\kappa) \in (0,2] \times (0,1/2)$, we have that, for any $n \in \N$, $0<(1-\kappa \mu)^n<1$, hence $T^n(x_0) \in \Int A$ and $T^n(x_0) \to P_L(x_0)$ as $n \to +\infty$.
\end{proof}

% \begin{figure}[!ht]
%     \centering
%     \includegraphics[width = 11cm]{pics/Ex2.png}
%     \caption{Illustration of Example~\ref{Ex2}.}
%     \label{fig:Ex2}
% \end{figure}

\begin{figure}[ht]
    \centering
    \includegraphics[width=0.462\textwidth]{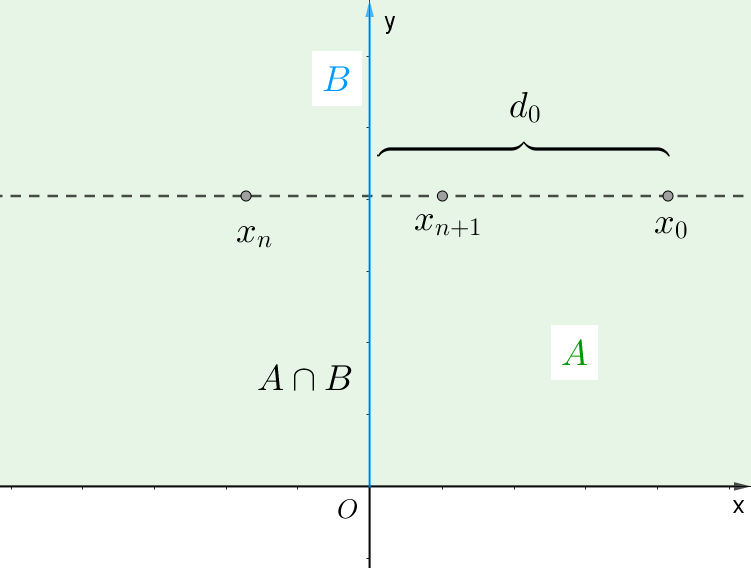}
    \hspace{0.08\textwidth}
    \includegraphics[width=0.4\textwidth]{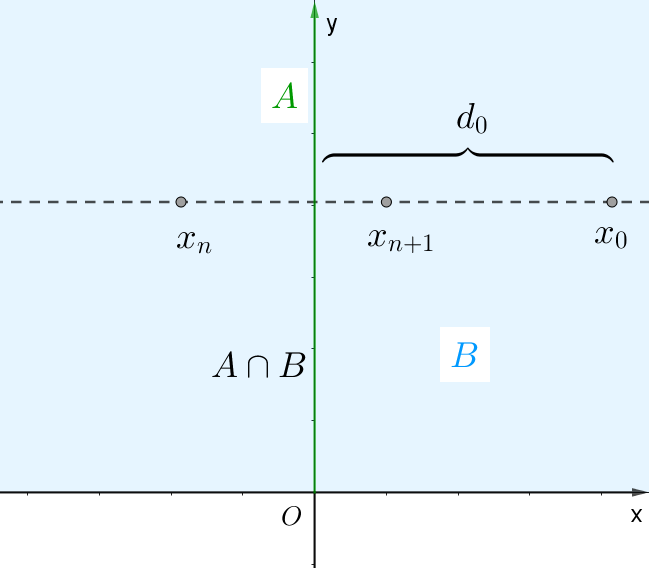}
    \caption{Left: Illustration of Example~\ref{Ex2}, when $(\lambda,\mu,\kappa) \in (0,2] \times (0,2] \times [1/2,+\infty) \; \text{ and } \; \kappa \mu \neq 1$. \\
    Right: Illustration of Example~\ref{Ex3}, when
    $(\lambda,\mu,\kappa) \in (0,2] \times (0,2] \times [1/2,+\infty) \; \text{ and } \; \kappa \lambda \neq 1$.}
    \label{fig:Ex2&3}
    % 0.45 0.4
\end{figure}

\begin{example}\label{Ex2}
Consider $A = \{(\rho , \theta)\with 0 \leq \theta \leq \pi\}$, $B = \{(\rho , \theta)\with \theta = \pi/2\}$, and \(x_0 \in \Int A\). Then for any $n \in \mathbb{N}$, $T^n(x_0) \notin \Fix T$ with $T =T_{\lambda,\mu}^{\kappa}$ and
\[(\lambda,\mu,\kappa) \in (0,2] \times (0,2] \times [1/2,+\infty) \quad \text{ and } \quad \kappa \mu \neq 1.\]
\end{example}
\begin{proof}
Obviously, $A\cap B = B$. Let $x_0 = (d_0,y)$ be the Cartesian coordinates of $x_0$, with $d_0\neq 0,y>0$. Then $d(x_0,B)=|d_0|$, see Figure~\ref{fig:Ex2&3}.
Since $P_A^{\lambda}(x_0) = x_0$, the exact same calculations~\eqref{eq:Ex1_1} and~\eqref{eq:Ex1_2} hold true and 
 again analogously we have for any $n \in \N$ that $T^n(x_0) = ((1-\kappa \mu)^n d_0,y)$.
For our choice of $(\mu,\kappa)$, it holds that for any $n \in \N$, $(1-\kappa \mu)^n \neq 0$, leading to $T^n(x_0) \notin \Fix T$. 
\end{proof}

% \begin{figure}[!ht]
%     \centering
%     \includegraphics[width = 11cm]{pics/Ex3.png}
%     \caption{Illustration of Example~\ref{Ex3}.}
%     \label{fig:Ex3}
% \end{figure}

\begin{example}\label{Ex3}
Let $A = \{(\rho , \theta)\with \theta = \pi/2\}$, $B = \{(\rho , \theta)\with 0 \leq \theta \leq \pi\}$, and \(x_0 \in \Int B\). Then for any $n \in \mathbb{N}$, $T^n(x_0) \notin \Fix T$ with $T =T_{\lambda,\mu}^{\kappa}$ and
\[(\lambda,\mu,\kappa) \in (0,2] \times (0,2] \times [1/2,+\infty) \quad \text{ and } \quad \kappa \lambda \neq 1.\]
\end{example}
\begin{proof}
    We know $A\cap B = A$. For $x_0 = (d_0,y)$, where $d_0\neq 0,y>0$, we have $d(x_0,A)=|d_0|$, as depicted in Figure~\ref{fig:Ex2&3}. Here,
    \[P_A^{\lambda}(x_0)=(1-\lambda)  (d_0,y)+\lambda (0,y) = ((1-\lambda) d_0,y) \; \in\; B.\]
    Hence $P_B^{\mu}(P_A^{\lambda}(x_0))=P_A^{\lambda}(x_0)$ and
    \[T(x_0)=(1-\kappa)x_0 + \kappa P_B^{\mu}( P_A^{\lambda}(x_0)) = (1-\kappa)  (d_0,y)+ \kappa  ((1-\lambda) d_0,y) = ((1-\kappa \lambda) d_0,y).\]
    Similarly, for any $n \in \N$, $T^n(x_0) = ((1-\kappa \mu)^n d_0,y).$
    For the specific choice of $(\lambda,\kappa)$, we have that, for any $n \in \N$, $(1-\lambda \mu)^n \neq 0$, thus $T^n(x_0) \notin \Fix T$. 
\end{proof}

\begin{figure}[ht]
    \centering
    \includegraphics[width=0.47\textwidth]{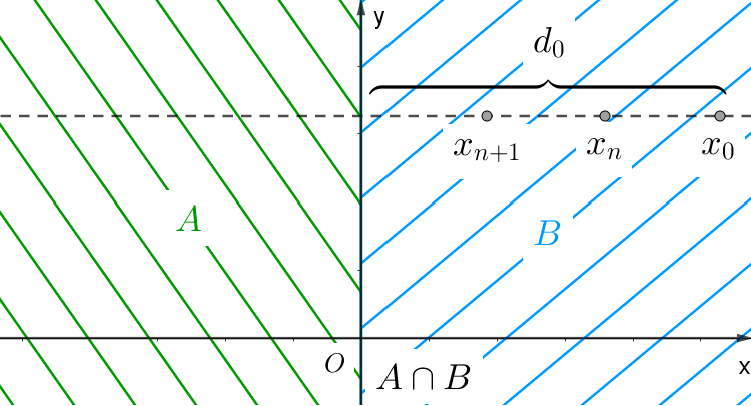}
    \hspace{0.05\textwidth}
    \includegraphics[width=0.43\textwidth]{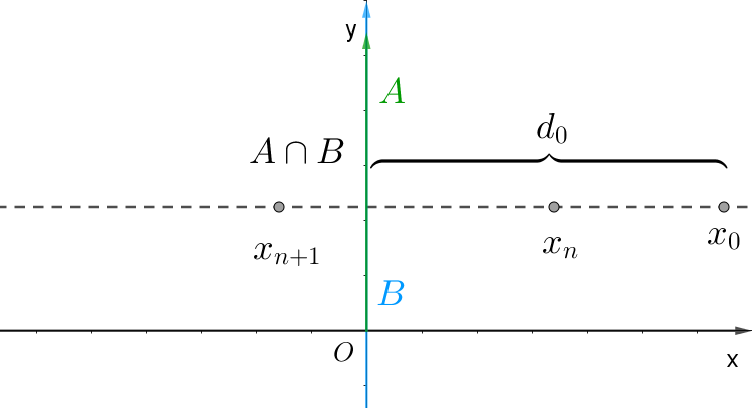}
    \caption{Left: Illustration of Example~\ref{Ex4}, when $(\lambda,\mu,\kappa) = (1/t,1/t,t), \; \text{where} \; 1/2 <t<1$. \\
    Right: Illustration of Example~\ref{Ex5}, when
    $(\lambda,\mu,\kappa) = (1/t,1/t,t), \; \text{where} \; t>1$.}
    \label{fig:Ex4&5}
\end{figure}

% \begin{figure}[!ht]
%     \centering
%     \includegraphics[width = 11cm]{pics/Ex4.png}
%     \caption{Illustration of Example~\ref{Ex4}.}
%     \label{fig:Ex4}
% \end{figure}

\begin{example}\label{Ex4}
Choose $A = \{(\rho , \theta)\with \pi/2 \leq \theta \leq 3\pi/2\}$, $B = \{(\rho , \theta)\with -\pi/2 \leq \theta \leq \pi/2\}$, and set \(x_0 \in \Int B\). Then for any $n \in \mathbb{N}$, $T^n(x_0) \notin \Fix T$ with $T =T_{\lambda,\mu}^{\kappa}$ and
\[(\lambda,\mu,\kappa) = (1/t,1/t,t), \quad \text{ where } \quad 1/2 <t<1.\]
\end{example}
\begin{proof}
Here, $A\cap B = \{(\rho , \theta)\with \theta = \pi/2 \; {\rm or}\; \theta = -\pi/2\}$. Let $L \coloneqq A \cap B$, and $x_0 = (d_0,y)$, where $d_0>0,y>0$, thus $d(x_0,L)=d_0$, see 
 Figure~\ref{fig:Ex4&5}.  We have
    \begin{align}\label{eq:Ex4_1}
    P_A^{\lambda}(x_0) = (1-1/t) \cdot (d_0,y)+1/t\cdot (0,y) = ((1-1/t)d_0,y).    
    \end{align} 
Then     
\begin{align}\label{eq:Ex4_2}
P_B^{\mu}(P_A^{\lambda}(x_0))=(1-1/t) \cdot ((1-1/t)d_0,y)+1/t\cdot (0,y) = ((1-1/t)^2 d_0,y),
    \end{align}  
and therefore
\begin{align}\label{eq:Ex4_3}
T(x_0) & = (1-\kappa)x_0 + \kappa P_B^{\mu}( P_A^{\lambda}(x_0))\notag \\
& = (1-t)  (d_0,y)+ t  (\left(1-1/t\right)^2 d_0,y) \\
& = \left(\left(1/t-1\right) d_0,y\right).\notag
\end{align}
Note, $ T(x_0)\in  \Int B$. Similarly, we obtain that, for any $n \in \N$, 
\[T^n(x_0) = ((1/t-1)^n d_0,y).\]
Since $1/2 <t<1$, we have for any $n \in \N$ that $0<(1/t-1)^n<1$. Thus, for any $n \in \N$, $T^n(x_0) \in \Int B$ and $T^n(x_0) \to P_L(x_0)$ as $n \to +\infty$. 
\end{proof}

% \begin{figure}[!ht]
%     \centering
%     \includegraphics[width = 11cm]{pics/Ex5.png}
%     \caption{Illustration of Example~\ref{Ex5}.}
%     \label{fig:Ex5}
% \end{figure}

\begin{example}\label{Ex5}
Let $A = \{(\rho , \theta)\with \theta = \pi/2\}$ and $B = \{(\rho , \theta)\with \theta = \pi/2 \; {\rm or}\; \theta = -\pi/2\}$. Suppose the Cartesian coordinates of $x_0$ are $x_0 = (d_0,y)$, where $d_0>0,y>0$. Then for any $n \in \mathbb{N}$, $T^n(x_0) \notin \Fix T$ with $T =T_{\lambda,\mu}^{\kappa}$ and
\[(\lambda,\mu,\kappa) = (1/t,1/t,t), \quad \text{ where } \quad t>1.\]
\end{example}
\begin{proof}
% $L\coloneqq\{(\rho , \theta)\with \theta = \pi/2 \; {\rm or}\; \theta = -\pi/2\}$
Obviously, $A\cap B = A$. Let $L\coloneqq A$, then $d(x_0,L)=d_0$, as illustrated in Figure~\ref{fig:Ex4&5}.
The same calculations \eqref{eq:Ex4_1}--\eqref{eq:Ex4_3} can be performed and again, we may claim for any $n \in \N$ that $T^n(x_0) = ((1/t-1)^n d_0,y)$.
For $t>1$, it holds that, for any $n \in \N$, $0<|1/t-1|^n<1$. Therefore, for any $n \in \N$, $T^n(x_0) \notin \Fix T$ and $T^n(x_0) \to P_L(x_0)$ as $n \to +\infty$. 
\end{proof}

\begin{figure}[!ht]
    \centering
    \includegraphics[width = 0.45\textwidth]{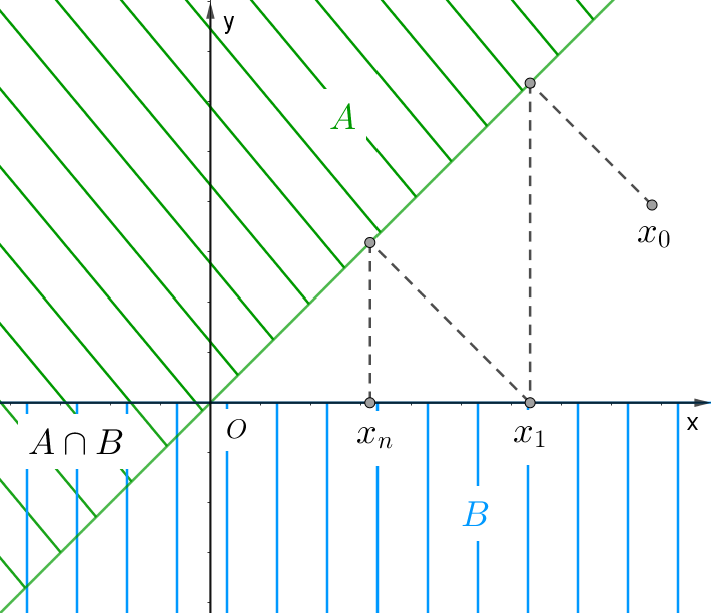}
    \caption{Illustration of Example~\ref{Ex6}, when $(\lambda,\mu,\kappa) = (1,1,1)$.}
    \label{fig:Ex6}
\end{figure}

\begin{example}\label{Ex6}
Let $A = \{(\rho , \theta)\with \pi/4 \leq \theta \leq 5 \pi/4\}$,  $B = \{(\rho , \theta)\with \pi \leq \theta \leq 2\pi \}$, and $x_ 0 = (x,y)$ be the Cartesian coordinates of $x_0 \in \R^2 \setminus (A\cup B)$. Then for any $n \in \mathbb{N}$, $T^n(x_0) \notin \Fix T$ with $T =T_{\lambda,\mu}^{\kappa}$ and
\[(\lambda,\mu,\kappa) = (1,1,1).\]
\end{example}
\begin{proof}
In this case, $T$ is the alternating projection operator, i.e., $T = P_B \circ P_A$. We see that
\[P_A(x_0)=\left( \frac{x+y}{2},\frac{x+y}{2} \right) \quad \mbox{and} \quad P_B(P_A(x_0))=\left( \frac{x+y}{2},0 \right).\]
Hence, for any $n \in \N$, $T^n(x_0) = \left((x+y)/2^n,0 \right) \notin \Fix T$ and $T^n(x_0) \to 0 \in \Fix T$ as $n \to +\infty$. 
\end{proof}

We conclude  Section~\ref{sec:counterexamples} with a formal proof of Theorem~\ref{thm:badmethods}. 

\begin{proof}[Proof of Theorem~\ref{thm:badmethods}] As described earlier in this section, the proof follows from the explicit constructions in Examples~\ref{Ex1}--\ref{Ex6}. Notice that union of the sets of the parameters values $(\lambda,\mu,\kappa)$ used in these examples together with the Douglas--Rachford selection $(2,2,1/2)$ comprises all possibilities given in the Definition~\ref{def:projop}. Since each demonstrates the existence of a pair of sets $A, B\subseteq\R^2$ and $x_0\in \R^2$ for which the method defined by the specific choice of the parameters does not converge finitely (while the set of fixed points of the corresponding operator is nonempty),  these examples together show the desired result. 
\end{proof}

\section{Conclusion and Outlook}

We have considered a reasonably large family of projection methods in a toy setting of two convex cones in the plane. We have shown that within this simple model, the Douglas--Rachford method outperforms all other methods from this class in terms of finite convergence. 

Even though this result appears new, and can be extended to the case of two convex polytopes in the plane, it is unclear if similar superiority in terms of finite convergence is typical for this method in other settings, in particular, for higher-dimensional and non-polyhedral problems. We note here that in the absence of Slater's condition, the Douglas--Rachford method may fail to have finite convergence in high dimensions even for polyhedral sets; see \cite[Example~8.8]{BD17}.

It is an exciting direction for future work to obtain a geometric characterization for the finite convergence of the Douglas--Rachford method for the nonpolyhedral cases, and for higher dimensions.

\subsection*{Acknowledgements} 

Part of this work was done during MND's visit to the University of New South Wales, Sydney. We thank UNSW Sydney for its hospitality. We are also grateful to the Australian Research Council for continuing support. Specifically, this research was partially supported by Discovery Projects DP200100124 and DP230101749.

\bibliography{refs}
\bibliographystyle{plain}

\end{document}